\documentclass[reqno]{amsart}
\usepackage{amssymb, amscd, amsfonts}
\usepackage[mathscr]{eucal}
\usepackage{amscd}

\newtheorem{theorem}{Theorem}[section]
\newtheorem{lemma}[theorem]{Lemma}
\newtheorem{corollary}[theorem]{Corollary}
\newtheorem{proposition}[theorem]{Proposition}

\theoremstyle{definition}
\newtheorem{definition}[theorem]{Definition}
\newtheorem{remark}[theorem]{Remark}

\newcommand{\restrict}{\,{\mathbin{\vert\mkern-0.3mu\grave{}}}\,}

\DeclareMathOperator{\McN}{\mathscr M}
\DeclareMathOperator{\McNn}{\mathscr M_{\it n}}
\DeclareMathOperator{\conv}{\rm conv}
\DeclareMathOperator{\den}{\rm den}
\DeclareMathOperator{\Zed}{\mathbb{Z}}
\DeclareMathOperator{\cube}{[0,1]^{\it n}}
\DeclareMathOperator{\maxspec}{\rm MaxSpec}

\title[Rational polyhedra and
projective unital $\ell$-groups]{\bf Rational
polyhedra and projective
lattice-ordered abelian groups with order unit}

\author[L.Cabrer]{Leonardo Cabrer $^\ddag$}
\address[L.Cabrer]{CONICET \\
Dep. de Matem{\'a}ticas -- Facultad de Ciencias Exactas \\
Universidad Nacional del Centro de la Provincia de Buenos Aires \\
Pinto 399 -- Tandil (7000) \\
Argentina }
\email{lcabrer@exa.unicen.edu.ar }

\author[D.Mundici]{Daniele Mundici$^\dag$}
\address[D.Mundici]{Dipartimento di
Matematica \, ``Ulisse Dini'' \\
Universit\`{a} degli Studi di Firenze \\
viale Morgagni 67/A \\
I-50134 Firenze \\
Italy}
\email{mundici@math.unifi.it }



\date{\today}
\begin{document}

\begin{abstract}
An {\it $\ell$-group} $G$ is an abelian group equipped with
a translation invariant lattice order.
Baker and Be\-y\-n\-on
proved that  $G$
is finitely generated projective iff it is finitely presented.
A {\it unital} $\ell$-group is an
$\ell$-group $G$
with a distinguished {\it order unit}, i.e.,
an element $0\leq u\in G$ whose positive
integer multiples eventually dominate every
element of $G$. While every finitely generated
projective unital $\ell$-group is finitely
presented, the converse does not hold in general.
Classical algebraic topology (\'a la Whitehead)
will be combined in this paper with the
W\l odarczyk-Morelli solution of the weak Oda
conjecture for toric varieties, to describe
finitely generated projective unital $\ell$-groups.
\end{abstract}

\maketitle

\section{Introduction: unital $\ell$-groups and rational polyhedra}
A  main reason of interest in  unital $\ell$-groups
is that Elliott classification yields a one-one
correspondence  $\kappa$
between
isomorphism classes of
 unital
AF $C^{*}$-algebras whose Murray-von Neumann
order of projections is a lattice, and
isomorphism classes of
countable unital $\ell$-groups:  $\kappa$  is
an order-theoretic enrichment of
Gro\-then\-dieck  $K_0$ functor,
\cite[3.9,3.12]{mun86}.

A unital $\ell$-group $(G,u)$ is {\it projective} if whenever
$\psi\colon (A,a)\to(B,b)$ is a surjective unital $\ell$-homomorphism
and $\phi\colon (G,u)\to(B,b)$ is a unital $\ell$-homomorphism, there
is a unital $\ell$-homomorphism $\theta\colon (G,u)\to(A,a)$ such that
$\phi= \psi \circ \theta$.
As usual,  {\it unital $\ell$-homomorphisms} between unital
$\ell$-groups are group
hom\-o\-mor\-phisms that also
preserve the order unit and the lattice
structure, \cite{bigkeiwol}.

 Baker \cite{bak} and Beynon
  \cite[Theorem 3.1]{bey77} (also see \cite[Corollary
 5.2.2]{gla}) gave the following characterization: {\it An $\ell$-group
 $G$ is finitely generated projective iff it is finitely presented.}
While the $(\Rightarrow)$-direction
 still holds for every unital
 $\ell$-group $(G,u)$  (see, e.g., \cite[Proposition 5]{mun08}),
 in this paper we will show that
 various arithmetical, geometrical and
 topological conditions must be imposed to ensure that
 a finitely presented $(G,u)$ is projective.

 For $n=1,2,\ldots$ we let ${\mathscr M}_{n}$ denote the
unital $\ell$-group of all continuous functions
$f\colon [0,1]^{n}\to
\mathbb R$ having the following property: there are linear
polynomials $p_{1},\ldots,p_{m}$ with integer coefficients, such that
for all $x\in [0,1]^{n}$ there is $i\in \{1,\ldots,m\}$ with
$f(x)=p_{i}(x)$.  ${\mathscr M}_{n}$ is equipped with the pointwise
operations $+,-,\max,\min$ of $\mathbb R$, and with the constant
function 1 as the distinguished order unit.

\begin{proposition}
\label{proposition:free}
{\rm (\rm \cite[4.16]{mun86})}   ${\mathscr
M}_{n}$ is generated by the coordinate maps
$\xi_ {i} \colon
[0,1]^{n}\to \mathbb R$
and  the order unit $1$.
For every unital $\ell$-group $(G,u)$ and
$0\leq g_{1},\ldots,g_{n}\leq u$,
if the set $\{g_{1},\ldots,g_{n}, u\}$
generates $G$
  there is a unique unital
$\ell$-homomorphism $\psi$ of ${\mathscr M}_{n}$
onto $G$ such that
$\psi(\xi_ {i})=g_{i}$ for each $i=1,\ldots,n.$
\end{proposition}

   An
   {\it ideal} $\mathfrak i$ of a unital
  $\ell$-group $(G,u)$ is
  the kernel of a unital $\ell$-ho\-m\-o\-mor\-ph\-ism
  of $G$, (\cite[p.8  and  1.14]{goo2}).
  $\mathfrak i$ is {\it principal} if
it
  is singly (=finitely)
  generated.

A unital $\ell$-group
$(G,u)$ is {\it finitely presented} if
for some $n=1,2,\ldots,$
$(G,u)$ is unitally $\ell$-isomorphic to the
quotient of  ${\mathscr M}_{n}$ by some principal ideal
$\mathfrak j$, in symbols,
$(G,u)\cong \McNn/\mathfrak j$.

For every nonempty closed set
$X \subseteq \cube$ we introduce the notation
\begin{equation}
    \label{equation:restrictions}
\McNn\restrict X=\{f\restrict X\mid f\in \McNn\}
\end{equation}
for the unital $\ell$-group of restrictions to $X$
of the functions in $\McNn$.

Following  \cite[1.1]{sta},
by a {\it polyhedron} $P$
  in $\cube$  we
mean a finite union of
(always closed)  simplexes $P=S_{1}\cup\cdots\cup S_{t}$
   in $\cube$. If the
 coordinates of the vertices of every simplex $ S_{i}$
  are rational numbers,  $P$ is said to be {\it rational}.

    For any rational point $v \in
  \mathbb R^{n}$ the least common denominator of
  the coordinates of $v$
  is called the \emph{denominator} of $v$  and is denoted $\den(v)$.

The relationship between rational polyhedra and
finitely presented
unital $\ell$-groups  is given by
the following result:

\begin{proposition}
    \label{proposition:finpres}
    {\rm (\cite[Propositions 4 and 5]{mun08})}
Let  $(G,u)$ be a  unital
    $\ell$-group.
    \begin{itemize}
    \item[(a)]   $(G,u)$ is
finitely presented
iff there is  $n=1,2,\ldots$
and   a rational polyhedron $P\subseteq \cube$ such that
$(G,u)$ is unitally $\ell$-isomorphic to $\McNn\restrict P$.
\item[(b)] If  $(G,u)$ is
finitely generated projective then
$(G,u)$ is finitely presented.
\end{itemize}
\end{proposition}

One may now
naturally ask for which rational polyhedra $P
\subseteq \cube$ the unital $\ell$-group
 $\McNn\restrict P$ is projective.
In   Theorem \ref{theorem:about}
and Corollary \ref{corollary:parafrasi}
it is shown that  $P$
satisfies the following
{\it necessary} conditions:
   (i) $P$ is contractible,
 (ii) $P$ contains a vertex of
 the $n$-cube $\cube$ and
 (iii)  $P$ has a regular
 triangulation $\Delta$
  (as defined in
Section \ref{section:segunda}
following   \cite{wlo})
such that for every maximal simplex $T\in \Delta$,
the denominators of  the vertices of $T$ are
coprime.
As proved in   Corollaries \ref{corollary:tree} and
\ref{corollary:parafrasi},
these three conditions
 are {\it sufficient} for  $\McNn\restrict P$
 to be  projective in case
 $P$   is
 one-dimensional.  Further, if
  $P$ is an $n$-dimensional rational polyhedron
  satisfying (ii) and (iii), then
   $\McNn\restrict P$  is a
  projective unital
 $\ell$-group, provided
Condition (i) is strengthened
  to the collapsibility (\cite{whi, ewa, sta})  of
 at least one  triangulation of  $P$.

 We refer to
\cite{bigkeiwol,fuc, gla, goo2}   for $\ell$-groups, to
\cite{hat} for algebraic topology,
 to  \cite{sta}  for polyhedral topology, and to
\cite{ewa}  for regular fans---the homogeneous
correspondents of rational polyhedra.
Their desingularization procedures yield  a
key tool
for our results.

\section{Regular triangulations, Farey mediants and blow-ups}
\label{section:segunda}

  For every (always finite)
  simplicial complex ${\mathcal K}$
   the
  point-set union of the simplexes of
  ${\mathcal K}$ is
  denoted $|{\mathcal K}|$;
  $\,\,\,{\mathcal K}$ is said to be
  a \emph{triangulation} of $|{\mathcal K}|$.
A simplicial complex is
  said to be a \emph{rational}  if the vertices of all
  its simplexes are rational.
 Given simplicial complexes ${\mathcal K}$ and ${\mathcal H}$
with $|{\mathcal K}|=|{\mathcal H}|$
we say that ${\mathcal H}$ is a
{\it subdivision} of ${\mathcal K}$
if every simplex of ${\mathcal H}$ is a union of
 simplexes of ${\mathcal K}$.
For any rational point $v \in
\mathbb R^{n}$ the
integer vector
\begin{equation}
    \label{equation:homogeneous}
 \widetilde{v } = \den(v)(v,1)\in \mathbb Z^{n+1}
 \end{equation}
is called the {\it homogeneous correspondent} of $v$.
An  $m$-simplex $\conv(w_{0}, \ldots,w_{m}) \subseteq [0,1]^{n}$
is said to be
\emph{regular} if its vertices are rational
and the set of integer vectors
$\{\widetilde{{w}_0}, \ldots, \widetilde{{w}_m}\}$
can be extended to a basis of
the free abelian group ${\mathbb Z}^{n+1}$.
Following \cite{wlo}, a simplicial complex
${\mathcal K}$
is
said to be a \emph{regular triangulation} (of
$|\mathcal K|$) if all
its simplexes are regular.
Regular triangulations are called ``unimodular''
in  \cite{mun08}.
Given a regular triangulation  $\Delta$ with $|\Delta| \subseteq \cube$,
the homogeneous correspondents of its vertices
are the  generating vectors of a complex
of  cones in $\mathbb R^{n+1}$, which is
a regular (also known as ``nonsingular'')
  fan \cite{ewa}.

\smallskip

\begin{lemma}
    \label{lemma:existence}
    {\rm (\cite[Proposition 1]{mun08})}
    For every rational polyhedron  $P$ there is a regular triangulation
$\Delta$ such that
 $P = |\Delta|$.
 \end{lemma}

\begin{lemma}
    \label{Lem-LinearMap} Let
$S=\conv(v_{1},\ldots,v_{k})\subseteq [0,1]^{n}$ be a regular
$(k-1)$-simplex and $\{w_{1},\ldots,w_{k}\}$ a set of rational points
in $[0,1]^{m}.$  Then the following conditions are equivalent:
\begin{itemize}
\item[(i)] $\den(w_{i})$ is a divisor of $\den(v_{i})$, for each
$i=1,\ldots,k$.
\item[(ii)]
For some
integer matrix $M\in\Zed^{m\times n}$ and  vector
$b\in\Zed^{m}$
we have $M v_{i}+b=w_{i}$ for each $i=1,\ldots,k$.
\end{itemize}
\end{lemma}

\begin{proof}
  For the nontrivial direction,
suppose  $\den(w_{i})$ is a divisor of $\den(v_{i})$, for each
$i=1,\ldots,k.$  With reference to (\ref{equation:homogeneous}),
let
 $\{\widetilde{v_{1}},\ldots,
     \widetilde{v_{k}},b_{k+1},\ldots,b_{n+1}\}$
     be a basis of the
     free abelian group $\Zed^{n+1}$,
     for suitable vectors
$b_{k+1},\ldots,b_{n+1}\in\Zed^{n+1}.$
 Let $D$  be the
     $(n+1)\times(n+1)$ integer
      matrix   whose columns are the vectors
     $\widetilde{v_{1}},\ldots,
     \widetilde{v_{k}},b_{k+1},\ldots,b_{n+1}$.
     Then
     $D^{-1}\in\Zed^{(n+1)\times (n+1)}.$
By hypothesis, for each  $i=1,\ldots,k$,
the vector
     $c_{i}=\den(v_{i})(w_{i},1)$ belongs to $\Zed^{m+1}.$
Let    $d_{k+1},\ldots,d_{n+1}$ be
vectors in $ \Zed^{m+1}$ such that
  for each $j=k+1,\ldots,n+1$
 the $(m+1)$th coordinate of
  $d_{j}$ coincides with the
  $(n+1)$th coordinate of $b_j$.
     Let  $C\in\Zed^{(m+1)\times (n+1)}$ be the matrix whose columns are given
     by the vectors $c_{1},\ldots,c_{k},d_{k+1},\ldots,d_{n+1}$.
Since the $(n+1)$th row of $D$ equals  the $(m+1)$th row of $C$,
we have
$$
       C D^{-1}=\left(\begin{tabular}{c|c}
       $M$ & $b$ \\
       \hline
       $0,\ldots,0 $ & $1$
       \end{tabular}\right)
     $$
     for some  $m\times n$ integer matrix $M$
     and integer vector $b\in\Zed^m$.
     For each  $i=1,\ldots,k$ we then have
     $
     (C D^{-1})\widetilde{v_{i}}=(C
     D^{-1})\den(v_{i})(v_{i},1)=\den(v_{i})(Mv_{i}+b,1).
     $
     By definition,
     $(C D^{-1})\widetilde{v_{i}}=c_{i}=\den(v_{i})(w_{i},1),$
    whence $Mv_{i}+b=w_{i}$ as desired.
\end{proof}

       \subsection*{Blow-up and Farey mediant}
Let $\Delta $  be a simplicial complex and
$p \in |\Delta|  \subseteq \mathbb R^{n}$.
Then the (Alexander, \cite{ale}) {\it blow-up ${\Delta}_{(p)}$
of $\Delta$ at $p$}
is the subdivision of $\Delta$ which is obtained
by replacing
every simplex $T \in \Delta$
that contains $p$
by the set
of all  simplexes of the form $\conv(F\cup\{p\})$, where
$F$ is any face of $T$ that does not contain $p$.
(We are using the terminology of \cite[p.376]{wlo}.
Synonyms of ``blow-up'' are ``stellar subdivision''
and
``elementary subdivision'',
\cite[III, 2.1]{ewa}.)

The notation
$\Delta_{(w_{1},\ldots,w_m)}$
stands for the final outcome of a
sequence of blow-ups of $\Delta$ at points
$w_{1},\ldots,w_m$, i.e.,
\begin{equation}
    \label{equation:iterated}
    \Delta_{(w_{1},\ldots,w_{t+1})}=
\Delta_{(w_{1},\ldots,w_{t})_{(w_{t+1})}}.
\end{equation}

For any regular  $m$-simplex $E =\conv(v_{0},\dots,v_m)
\subseteq \mathbb R^{n}$,
the {\it Farey mediant} of (the vertices of)
$E$ is the rational point $v$ of
$E$ whose homogeneous correspondent $\tilde v$ equals
$\widetilde{{v}_0}+\cdots+\widetilde{{v}_m}$.
This is in agreement with the classical terminology in
case $E=[0,1]$.
If $E$ belongs to a regular triangulation
$\Delta$ and $v$ is the Farey mediant of $E$ then the blow-up
$\Delta_{(v)}$ is a regular triangulation. Also the
converse is true,
(a proof can be obtained from \cite[V, 6.2]{ewa}).
  $\Delta_{(v)}$  will
be called the {\it Farey blow-up} of $\Delta$ at $v$.
By a {\it (Farey) blow-down} we understand  the inverse
operation of a (Farey) blow-up.

\smallskip
The proof of
the ``weak Oda conjecture'' by
W\l odarczyk \cite{wlo} and Morelli \cite{mor}
immediately yields:

\begin{lemma}
    \label{lemma:weakoda}
    Let $P$ be a rational polyhedron.
    Then any two regular triangulations
 of $P$ are connected by a finite path
    of Farey blow-ups and Farey blow-downs.
\end{lemma}

\begin{definition}
    \label{definition:strongly}
A   triangulation $\Delta$  of a
rational polyhedron $P\subseteq \cube$ is
said to be {\it strongly regular}  if
it is regular and
the greatest common divisor of the denominators of the
vertices of each maximal simplex
of $\Delta$  is equal to $1$.
\end{definition}

\begin{lemma}
    \label{Lem-BlowupPreserGCD}
  Let  $\Delta$ and $\nabla$ be regular triangulations
  of a rational polyhedron $P\subseteq\cube$.
  Then $\Delta$ is strongly regular iff $\nabla$ is.
\end{lemma}

\begin{proof}
In view of
  Lemma \ref{lemma:weakoda} it is enough to
argue in case
$\Delta$ is the blow-up at
the Farey mediant $v$
of an  $m$-simplex
$S=\conv(v_0,\ldots,v_m)\in\nabla$.
Let $M\in\nabla$ be a maximal $(m+k)$-simplex
such that $S\subseteq M$. There
are $w_1,\ldots,w_k\in M$ such that
$M=\conv(v_0,\ldots,v_m,w_1,\ldots,w_k)$.
Since $\den(v)$ is equal to  $\sum_{j=0}^m \den(v_j)$,
the greatest common divisor of  the integers
$\den(v_0),\ldots,\den(v_m),\den(w_1),\ldots,\den(w_k)$
coincides with the greatest common divisor of
$$\den(v_0),\ldots,\den(v_{i-1}),
\den(v),\den(v_{i+1}),\ldots,\den(v_m),\den(w_1),
\ldots,\den(w_k).$$
\end{proof}

    \begin{lemma}
	\label{lemma:bigden}
	If $T=\conv(v_1,\ldots,v_{t})
	\subseteq \cube$
	is a regular
	$(t-1)$-simplex and the denominators
of its vertices
	are coprime, then for all large integers $l$ there is
	a rational point $v \in T$ such that $\den(v)$ is a divisor
	of $l$.
	\end{lemma}

\begin{proof} Let $\widetilde{{v_1}},\ldots,\widetilde{v_t},
w_{t+1},\ldots,w_{n+1} $ be a basis $\mathcal B$ of the free abelian
group $\Zed^{n+1}$.  For each $i=1,\ldots,t,$ let $d_i=\den(v_i)$.
Since $\gcd(d_{1},\ldots,d_t)=1$,
without loss of generality the $(n+1)$th coordinate
of each vector  $w_j$ can be assumed to be $0$.  Let further
$$
C=\mathbb{R}_{\geq 0}\widetilde{v_1}+\cdots+ \mathbb{R}_{\geq 0}\widetilde{v_t} +
\mathbb{R}_{\geq 0}w_{t+1}+ \cdots \mathbb{R}_{\geq 0}w_{n+1}
$$
denote the {\it
cone} positively spanned by
$\mathcal B$ in  the vector space $\mathbb R^{n+1}$.
Let the vector
$s=(s_1,\ldots,s_{n+1})\in\Zed^{n+1}$ be defined
by $s=\widetilde{v_1}+\cdots,\widetilde{v_t}+
w_{t+1}+\cdots+w_{n+1}$.
Let $\mathbb{R}_{\geq 0}\,s$ denote the
{\it ray} of $s$, i.e., the positive real span of
the vector
$s$ in $\mathbb R^{n+1}.$
For every integer $l=1,2,\ldots$,
let the hyperplane $H_{l}$ be defined by
$$H_{l}=\{(y_1,\ldots,y_{n+1})
\in \mathbb R^{n+1}\mid y_{n+1}=l\}.$$
The vanishing of the last coordinate of each
$w_j$ is to the effect that
$s_{n+1}=d_1+\cdots+d_k>0,$ whence
the set $H_l\cap\mathbb{R}_{\geq 0}\,s$
contains a single point, denoted $h_l$.
This  is a rational point lying
in the interior of $C$.
In particular,
 for some $0<\epsilon\in\mathbb R\,\,$
 the point
$h_1$ lies in a closed $n$-cube
 of side length $\epsilon$
  contained in $C\cap H_1$.
Consequently,  for all
large integers $l$, the rational point
$h_l$ lies in some closed unit $n$-cube $D_l$
contained in the convex set $H_l\cap C.$
Necessarily $D_l$ contains an integer
point $p=(p_1,\ldots,p_n,l)$.

To conclude the proof, as
  noted in \cite[V, 1.11]{ewa},
there are integers $m_1,\ldots,m_{n+1} \geq 0$ such that
$p=m_1\tilde{v_1}+\cdots+m_{t}
\tilde{v_t}+m_{t+1}w_{t+1}+\cdots+m_{n+1}w_{n+1}.$
Let the   vector $q\in \mathbb Z^{n+1}$ be defined by
$
q=m_1\tilde{v_1}+\cdots+m_{t}\tilde{v_t}.
$
Since the $(n+1)$th coordinates
of $p$ and of $q$ are  equal,
$m_1d_1+\cdots+m_td_t=l$.
Let $v$ be the only
rational point of $\cube$ whose homogeneous correspondent $\tilde v$
lies on the ray  $\mathbb R_{\geq 0}\,q$ of $q$.
Then $v$ belongs to $T$, and $r$ is a
positive integer multiple of $\tilde v.$  Thus the
$(n+1)$th coordinate $\den(v)$
of $\tilde v$    is a divisor of  the $(n+1)$th coordinate  $l$
of  $q$.
  \end{proof}

Our next result essentialy follows from
Cauchy's 1816 analysis of the Farey sequence,
({\it Oeuvres}, II S\'erie, Tome VI, 1887, pp.146--148,
or Tome II, 1958, pp.207--209),
and is also a consequence of the
De Concini-Procesi theorem on elimination
of points of indeterminacy,
\cite[p.252]{ewa}.  For the sake of completeness
we give the elementary proof here:

\begin{proposition}
    \label{proposition:cauchy}
    If $T\subseteq \cube$ is a regular simplex
    then for every rational point  $v\in T$
    there is a sequence of regular complexes
    $
    \Delta_0=\{T\mbox{ and its faces} \},\,
    \Delta_1,\ldots,\Delta_u
    $
    such that $\Delta_{i+1}$ is a Farey blow-up of
    $\Delta_{i}$,  and $v$ is a vertex of
    (some simplex of) $\Delta_u$.
    \end{proposition}

\begin{proof}
Let $\omega$ be a fixed but otherwise arbitrary
 well-ordering of the set of all pairs of distinct
rational points  (=edges) in $\cube$.
We now inductively
define the regular triangulation $\Delta_{i+1}$ of $T$ by

\begin{quote} $\Delta_{i+1}=$ the  blow-up of $\Delta_{i}$ at the
Farey mediant of the $\omega$-first edge $\conv(w_1,w_2)$ of $\Delta_{i}$
such that $\den(w_1)+\den(w_2)\leq \den(v)$.
\end{quote}

\noindent This
sequence must terminate after a finite number $u$ of steps, just because
there are only finitely many rational points $w$ in $\cube$ satisfying
$\den(w) \leq \den(v)$.  Let $F$ be the {\it
smallest} simplex of $\Delta_u$ containing $v$.  In other words,
$F$ is the
intersection of all simplexes of $\Delta_u$ containing $v$.  It
follows that $v$ belongs to the relative interior of $F$.
By way of contradiction,  suppose  $v$ is not a vertex of $F$.
Then, for some
$w_1,\ldots,w_r\in \cube$
with $r\geq 2$, we have
$F=\conv(w_1,\ldots,w_r)$  and
  $\den(v)\geq \den(w_{1})+\cdots+\den(w_{r})$.
The inequality is strict, unless $v$ is the Farey mediant of
$F$.  (See e.g., \cite[ V, 1.11]{ewa}.)
A fortiori,  $\den(v)\geq \den(w_{1})+\den(w_{2})$, whence
the Farey blow-up $\Delta_{u+1}$ of $\Delta_{u}$ exists,
against our assumption about  $u$.
    \end{proof}

\section{$\mathbb Z$-retracts and
projective unital $\ell$-groups}

Given rational polyhedra $P\subseteq [0,1]^{n}$ and $Q\subseteq
[0,1]^{m}$ together with a map $\eta \colon P \rightarrow Q$, we say
that $\eta$ is a {\it $\Zed$-map}
if there is
a   triangulation $\mathcal K$  of $P$ such that over every
simplex $T$  of  $\mathcal K$,    $\eta$
coincides with a linear map
$\eta_{T}$  with
integer coefficients.

Since the intersection
of any two simplexes of
$\mathcal K$ is again a
(possibly empty)  simplex of
$\mathcal K$,  the continuity of
$\eta$ follows automatically.
The assumed properties
of the finite set of maps
$\{\eta_T\mid T\in \mathcal K\}$,
jointly with
the rationality of $P$, are to the effect
that
$\mathcal K$ can be assumed rational,
without loss of generality.
It follows that
 $\eta(P)$ is
a rational polyhedron in $[0,1]^m.$

\smallskip
A $\Zed$-map $\theta \colon P \rightarrow Q$ is said to be a {\it
$\Zed$-homeomorphism} (of $P$ onto $Q$)
if it is one-one onto $Q$ and the inverse
$\theta^{-1}$ is a $\Zed$-map.

A $\Zed$-map $\sigma\colon P\rightarrow P$ is a {\it
$\Zed$-retraction of P} if it is idempotent,
$\sigma\circ\sigma = \sigma.$
The rational polyhedron
$R=\sigma(P)\subseteq \cube$ is said to be a
$\Zed${\it-retract of} $P$.

\medskip

  If  $U, V, W$  are rational polyhedra in $\cube$,
$\mu$ is a $\Zed$-retraction of $U$ onto $V$,
and
$\nu$ is a $\Zed$-retraction of $V$ onto $W$,
then the
composite map  $\nu\circ\mu$ is a
$\Zed$-retraction of $U$ onto $W$.

\medskip
The relationship between $\mathbb Z$-retracts of cubes and
finitely generated projective
unital $\ell$-groups is given by the following

\begin{theorem}
\label{Thm_Projective_Retraction} A  unital
$\ell$-group $(G,u)$ is
finitely generated projective iff it is
unitally $\ell$-isomorphic
to $\McNn\restrict P$ for some  $n=1,2,\ldots$ and some
$\mathbb Z$-retract $P$  of $\cube$.
\end{theorem}

\begin{proof}
In \cite[3.9]{mun86} a
categorical equivalence  $\Gamma$
is established between
unital $\ell$-groups and {\it MV-algebras}---those
algebras satisfying the same $(\oplus,\neg)$-equations
as  the unit interval $[0,1]$ equipped
with truncated addition  $x\oplus y = \min(x+y,1)$
and involution $\neg x=1-x$.  By definition,
$\Gamma(G,u)=
\{g\in G \mid 0\leq g\leq u\}$.
Further, for every unital $\ell$-homomorphism
$\theta\colon (G,u)\to (G',u')$,
$\Gamma(\theta)$ is the restriction of $\theta$ to
 $\Gamma(G,u)$.
The preservation properties of
$\Gamma$ are to the effect that
$(G,u)$ is finitely generated
projective iff so is  $\Gamma(G,u)$,
(see \cite[3.4, 3.5]{mun86}).
Now apply
\cite[Theorem 1.2]{cabmun}.
\end{proof}

\smallskip

Let, as above,
$P\subseteq [0,1]^{n}$ and $Q\subseteq
[0,1]^{m}$ be
rational polyhedra, together with a $\Zed$-map
$\eta\colon P\to Q.$  Then for
every rational point $v\in P$,
\begin{equation}
\label{equation:denominators} \den(\eta (v))\,\, \mbox{ is a divisor
of }\,\, \den(v).  \end{equation}
Conversely, we have

\begin{lemma}
\label{lemma-ExtensionsToZmorphisms}
Let $P\subseteq \cube$ be a
rational polyhedron, $\Delta$  a regular triangulation of $P$, and
$\mathcal{V}$ the set of vertices of $\Delta$.  Let the map
$f\colon\mathcal{V}\rightarrow [0,1]^m$ be such that
$\den(f(v))$ is a divisor of
$\den(v)$ for every $v\in\mathcal{V}$.
Then $f$
can be uniquely extended to a
$\Zed$-map $\eta\colon P\rightarrow [0,1]^m$ which is linear on
each simplex of $\Delta$.  \end{lemma}

\begin{proof} By  Lemma \ref{Lem-LinearMap}, for each
$S\in\Delta$ there is a linear map with integer coefficients
$\eta_{S}\colon S\rightarrow[0,1]^{m}$ such that $ \eta_S(v)=f(v).$
The uniqueness of each $\eta_S$ ensures $\eta=
\bigcup\{\eta_S\mid S\in\Delta\}$ is well defined. Since $\eta$
coincides with $\eta_S$ over every simplex $S\in T$,
it is the desired $\Zed$-map.  \end{proof}

\smallskip The following result states that
the property of being a
$\Zed$-retract of some cube  is invariant under
$\Zed$-homeomorphisms:

\begin{lemma}
    \label{Lem-Zhomeo-Retrac}
    Let $\eta\colon[0,1]^{n}\rightarrow P$ be a
    $\Zed$-retraction onto $P$,
    and $\theta\colon P\rightarrow Q\subseteq[0,1]^{m}$  a
    $\Zed$-homeomorphism of $P$ onto $Q$. Then
    $Q$ is a $\Zed$-retract of $[0,1]^{m}$.
\end{lemma}

\begin{proof}  We first prove the following

    \medskip
    \noindent {\it Claim.}  There is a  regular
    triangulation $\Delta$ of $[0,1]^m$ such that  the set
    $\Delta_{Q}=\{T\in \Delta\mid T\subseteq Q\}$ is a triangulation
    of $Q$,
    and $\theta^{-1}$ is linear over each
simplex of $\Delta_{Q}$.

\medskip
As a matter of fact, let $\mathcal K$ be a rational  triangulation of
$Q$ such that
the $\mathbb Z$-map
$\theta^{-1}$ is linear over each
simplex  of $\mathcal K$.
Then   the affine counterpart of
\cite[III, 2.8]{ewa}
provides  a rational triangulation
$\nabla$  of $[0,1]^m$ such that
$\mathcal K \subseteq \nabla$.
The desingularization
procedure of \cite[Theorem 1.2]{mun88} now yields
a regular subdivision
$\Delta$ of $\nabla$ having the desired properties
to settle our claim.

\bigskip

    Let $o$ denote the origin of $\mathbb R^n$.
By   Lemma \ref{lemma-ExtensionsToZmorphisms} we have a uniquely determined
    $\Zed$-map $\mu\colon[0,1]^m\rightarrow[0,1]^n$ satisfying
    $$ \mu(v)=
    \left\{ \begin{tabular}{ll} $\theta^{-1}(v)$ & if $v\in Q $\\ $o$ & if
    $v\not\in Q $ \end{tabular} \right.
    $$
    for each  vertex $v$ of
    $\Delta$.  By definition, $\mu\restrict Q=\theta^{-1}$, whence
    $\mu(Q)=P$.  From $\eta\restrict P=Id_{P}$ it follows that
    $\theta\circ\eta\circ\mu\restrict Q=Id_{Q} $ and $
    \theta\circ\eta\circ\mu([0,1]^{m})=\theta(P)=Q. $ In conclusion,
    the map
    $\theta\circ\eta\circ\mu\colon[0,1]^{m}\rightarrow Q$ is a
    $\Zed$-retraction onto $Q$.
\end{proof}

\begin{lemma}
    \label{lemma:existMultiple}
   Let $\conv(v,w)\subseteq\cube$ be a regular
   $1$-simplex such that $a=\den(v)$ and $b=\den(w)$ are coprime.
   Then for each integer $m>0$ there is a rational point
   $z\in\conv(v,w)$ such that $m$ is a divisor of $\den(z)$.
\end{lemma}

\begin{proof}
  By hypothesis, there exist integers $p,q$ satisfying
  \begin{itemize}
     \item[(a)] $qa-pb=1$,
     \item[(b)] $0\leq p<a$ and $0< q\leq b$.
  \end{itemize}
  By (a), the two vectors  $(p,a)$ and
  $(q,b)$ form a basis of
  $\Zed^2$.  Stated otherwise,
  $[p/a,q/b]$ is a regular
  1-simplex.  By  (b),  $[p/a,q/b]\subseteq [0,1]$.
  Again by (a), $p$ and $a$ are coprime,
  whence $\den(p/a)=a=\den(v)$. Similarly,
  $\den(q/b)=\den(w)$.
  Lemma \ref{lemma-ExtensionsToZmorphisms}
  now yields a $\Zed$-homeomorphism $\eta$ of  $[p/a,q/b]$
  onto $\conv(v,w)$. Let $s\in[p/a,q/b]$ be a rational
  point such that $m$ is a divisor of $\den(s)$.
  A trivial density argument shows the existence
  of $s$.
By (\ref{equation:denominators}),  $m$ is a
  divisor of $\den(\eta(s))=\den(s)$
  and
  $\eta(s)$  is the desired rational point of $ \conv(v,w)$.
\end{proof}

\begin{theorem}
    \label{theorem:about}
If  the   polyhedron
    $P$ is a $\mathbb Z$-retract of $\cube$ then
\begin{itemize}
\item[(i)]
$P$ is contractible,

\item[(ii)] $P$ contains a vertex of $\,\,\cube$, and

\item[(iii)]
$P$ has a strongly regular triangulation  (Definition
\ref{definition:strongly}).
\end{itemize}
    \end{theorem}

    \begin{proof}  The proof of
  (i) is a routine exercise in algebraic topology,
  showing that $\cube$ is contractible,
  and a retract of a contractible space
  is contractible.

Concerning
  (ii), let $\eta\colon \cube\to P$ be a
  $\mathbb Z$-retraction onto $P$.
  By
  (\ref{equation:denominators}),  $\eta$ must
  send every vertex of $\cube$  into some
  vertex of $\cube.$

To prove (iii), let $\Delta$ be a regular triangulation of $P$, as
given by   Lemma \ref{lemma:existence}. Let
the $r$-simplex  $T =\conv(v_0,\ldots,v_r)$
be  maximal in $\Delta$.  Let us
write $d=\gcd(\den(v_0),\ldots,\den(v_r))$, with the intent of proving
$d=1$.  Let $z$ be a rational point lying in the relative interior of
$T$, say for definiteness $z=$ the Farey mediant of $T$.  Since $T$ is
maximal there is an open set $U\subseteq \cube$ such that $z\in U$ and
$U\cap P\subseteq T.$ Let $\eta\colon \cube\to P$ be a $\mathbb
Z$-retraction onto $P$.  Then $\eta^{-1}(U)$ is an open set.  Let $w $
be a rational point in $\eta^{-1}(U)$
whose denominator is a prime $p >    d.$
Since $\eta(w)$ lies in the regular simplex $T$, by
Proposition \ref{proposition:cauchy} $\eta(w)$ can be obtained via a
finite sequence of Farey blow-ups starting from the regular complex
given by $T$ and its faces.  One immediately verifies that $d$ is a
common divisor of the denominators of all Farey mediants thus
obtained.  In particular, $d$ is a divisor of $\den(\eta(w))$.  Since
$p$ is prime, from (\ref{equation:denominators}) it follows that
$\den(\eta(w))\in \{1,p\}$.  Since $p>d$, $d=1$.  We have proved that
$\Delta$ is strongly regular.
\end{proof}

\begin{remark}
    \label{remark:trebis}

    By   Lemmas
  \ref{lemma:existence}
    and \ref{Lem-BlowupPreserGCD},
Condition (iii)  above is equivalent
to
\begin{itemize}
    \item[(iii')] {\it Every regular triangulation
    of $P$ is strongly regular.}
    \end{itemize}
    \end{remark}

\medskip

 Condition (i)  has the following
equivalent reformulations  (for definitions
see the references given in the proof):

\begin{proposition}
    \label{proposition:various}
    For $P\subseteq \cube$ a rational
polyhedron, the following conditions are
equivalent:

\begin{itemize}
\item[$(\alpha)$]  $P$ is contractible.

\item[$(\beta)$]  $P$ is $n$-connected,
 i.e., the homotopy group
 $\pi_{i}(P)$ is trivial for each $i=0,\ldots,n$.

\item[$(\gamma)$]  $P$ is a deformation retract of
$\cube$.

\item[$(\delta)$]
$P$ is a retract of
$\cube$.

\item[$(\epsilon)$]   $P$ is an absolute retract for the class
  of metric spaces.
\end{itemize}
    \end{proposition}

\begin{proof}
        $(\alpha)\Rightarrow(\beta)$, \cite[p.405]{spa}.

	$(\beta)\Rightarrow(\alpha)$,
	\cite[p.359]{hat}.

	     $(\alpha)\Rightarrow(\gamma)$
	     is a consequence of Whitehead theorem,
	     \cite[346]{hat}.

	     $(\gamma)\Rightarrow(\delta)$, trivial.

 $(\delta)\Rightarrow(\alpha)$, because a retract of a contractible
 space  (like  $\cube$)
  is contractible.

    $(\alpha)\Leftrightarrow(\epsilon)$,
    \cite[15.2]{dug} together with \cite[p.522]{hat}.

\end{proof}

\section{A converse of Theorem \ref{theorem:about}}

In    Theorem \ref{theorem:collapsible} below we will prove that
Conditions (ii) and (iii) of   Theorem \ref{theorem:about},
together with a stronger form of Condition (i), known as
collapsibility \cite[p.97]{ewa}, \cite[6.6]{sta},
are also sufficient for a polyhedron
$P\subseteq\cube$ to be a
$\Zed$-retract of $\cube$.
The necessary  notation and terminology
are as follows:

An $m$-simplex $T$ of a simplicial complex
$\nabla$ in $[0,1]^n$ is said to have a {\it free face}
$F$ if $F$ is a {\it facet}
(=maximal proper face)
of $T$, but is a face of no other
$m$-simplex of $\nabla$.
It follows that
$T$ is a maximal simplex of $\nabla$,
and the
removal from $\nabla$ of both $T$ and $F$
results in the subcomplex $\nabla'=
\nabla\setminus\{T,F\}$ of $\nabla$.
The transition  from $\nabla$ to
$\nabla'$ is called an {\it elementary collapse}
in \cite[III,7.2]{ewa}
(``elementary contraction'' in \cite[p.247]{whi}).
If a simplicial complex $\Delta$ can be obtained
from $\nabla$ by a sequence of elementary collapses
we say that $\nabla$  {\it collapses to}  $\Delta$.
We say that $\nabla$ is {\it collapsible} if it
collapses to (the simplicial
complex consisting of) one of its vertices.

\bigskip
Given a rational polyhedron
$P\subseteq\cube$ and  a point
$a\in\cube$, following tradition
we let
\begin{equation}
\label{equation:join}
aP=\bigcup\{\conv(a,x)\colon x\in P\},
\end{equation}
and we say that
$aP$ is the {\it join} of $a$ and $P$.
If $P=\emptyset$   we let
$aP=a$.
Further, for
 any simplex  $S$   we use the notation
\begin{equation}
    \label{equation:dot}
\dot{S}=\bigcup\{F\subseteq S\mid F \text{ is a facet of } S\}
\,\,\,\mbox{ and }\,\,\,
a\dot{S} = \bigcup\{aT\mid T \mbox{ a facet of $S$}\}.
\end{equation}
Finally, we denote by $o$ the origin, and by
\begin{equation}
    \label{equation:basis}
e_1,\ldots,e_n
\end{equation}
the standard basis vectors
in $\mathbb R^n$.

\begin{lemma}
\label{Lem:empty-m-ahedron} Let $m_1,\ldots,m_n$ be coprime integers
$\geq 1$ and $s\in \{2,\ldots,n\}$.  Let
$$M=\conv(e_1/m_1,\ldots,e_n/m_n),\,\,\,
F=\conv(e_1/m_1,\ldots,e_{s-1}/m_{s-1}),\,\,\,p=e_{s}/m_{s}.$$ Then
there is a $\Zed$-retraction   of $M\cup o(pF)$ onto $M\cup
o(p\dot{F})$.  \end{lemma}

\begin{proof}
First of all,   both simplexes  $M$ and
$o(pF)$ are regular.
In the light of  Lemma \ref{lemma:bigden},
let us fix
an integer
$k \geq 1$ such that for every integer $l \geq k$ there exists a point
in $M$ whose denominator is a divisor of  $l$.

Let   the regular triangulation  $\Phi$  of $oF$ consist
 of  $oF$ together with its faces.
Let  $t_1$ be the Farey mediant
of $oF$.
Proceeding inductively, for each
$i=1,\ldots,k$ let $t_{i+1}$
be the Farey mediant of $t_{i}F$. Let $\Psi=
\Phi_{(t_1,\ldots,t_{k+1})}$.
By construction,
\begin{equation}
    \label{equation:inequalities}
\den(t_{k+1})=1+(k+1)\sum_{i=1}^{s-1} m_i >
\den(t_k)= 1+ k\sum_{i=1}^{s-1} m_i >  k.
\end{equation}
Since the   1-simplex $\conv(t_k,t_{k+1})$
satisfies the hypotheses
of  Lemma \ref{lemma:existMultiple},
there is a rational point
$p^*\in\conv(t_k,t_{k+1})$ such that
 $m_s$ is divisor of
$\den(p^*)$, in symbols,
\begin{equation}
    \label{equation:useful}
   m_s = \den(p) \,\,\, | \,\,\,\den(p^*).
    \end{equation}

Let $\Lambda$ be
the regular triangulation consisting of the  1-simplex
$\conv(t_k,t_{k+1})$ together with its faces.
By   Proposition
\ref{proposition:cauchy} there is a
finite sequence of Farey blow-ups
$\Lambda, \Lambda_{(w_1)},
\Lambda_{(w_1,w_2)},\ldots,\Lambda_{(w_1,\ldots,w_u)}
$
such that
$p^*$ is a vertex of some simplex of
$\Lambda_{(w_1,\ldots,w_u)}$.
The sequence of  consecutive Farey blow-ups of $\Psi$
 at  $w_1,\ldots,w_u$ yields the
regular triangulation
$\Delta=\Psi_{(w_1,\ldots,w_u)}$  of the
polyhedron  $oF.$
Let $w$ be an arbitrary point in the set
$V=\{w_1,\ldots,w_u, t_{k+1}\}$.
By  (\ref{equation:inequalities})
we can write
$
    \den(w) > \den(t_k) >k,
$
whence  our initial stipulation
about  the integer $k$ yields
a rational point  $x_w\in M$ such that $\den(x_w)$
is a divisor of $\den(w)$, in symbols,
\begin{equation}
    \label{equation:rosso}
    \forall w\in V\,\,\exists x_w\in M
    \mbox{ such that } \den(x_w) \,|\, \den(w).
 \end{equation}

To conclude the proof,
let   the regular triangulation
$\nabla$   of $M\cup o(pF)$ consist of all faces of $M$
together with the set of simplexes
$
\{pS\mid S\in \Delta\}\cup\Delta.
$
The vertices of $\nabla$  are
$
e_1/m_1,\ldots,e_{s-1}/m_{s-1},\,\, e_s/m_s=p,\,\,
e_{s+1}/m_{s+1},\ldots,e_{n}/m_{n}\in M,
$
together with
$
o,t_1,\ldots,t_{k+1}\in oF \text{ and } w_1,\ldots,w_u
\in\conv(t_k,t_{k+1}).
$
Let $\xi$ be the unique continuous  map
of $M\cup o(pF)$ into $\cube$ which is linear over
each simplex of $\nabla$, and for each  vertex $v$ of $\nabla$
satisfies

\begin{equation}
    \label{equation:casistica}
  \xi(v)=\left\{\begin{tabular}{ll}
    $p$& if $v=p^*$  \\
  $o$& if $v\in \{t_{1},\ldots,t_k\}$ or
  $v\in\conv(t_k,p^*)\setminus\{p^*\}$\\

  $x_v$& if $v\in\conv(p^*,t_{k+1})\setminus\{p^*\}\,$ \\
  $v$& if $v\in M\cup \{o\}.$\\
  \end{tabular} \right.
\end{equation}
By    Lemma
\ref{lemma-ExtensionsToZmorphisms},
recalling (\ref{equation:useful}) and (\ref{equation:rosso}),
$\xi$ is a  $\Zed$-map.
For some (uniquely determined)
permutation $\beta$  of $\{1,\ldots,u\}$
let the list
$
 t_{k}, w_{\beta(1)},\dots,
 w_{\beta(u)}, t_{k+1}
$
display the vertices of $\nabla$
lying on $\conv(t_k,t_{k+1})$
in the order of increasing distance from
$t_k$.  The maximal simplexes of
$\nabla$
are $M,\,\,\,p(t_{k+1}F)$ and $p(\conv(a,b,S))$,  where $S$
ranges over the set $\dot{F}$
of  facets of $F$,
and, $a,b\in oF$ are consecutive vertices in the list
$$
o, t_1, t_2,\ldots,t_k, w_{\beta(1)},\dots,
 w_{\beta(u)}, t_{k+1}.
$$
For any such maximal
simplex $T\in \nabla$,  a
tedious but straightforward
perusal of (\ref{equation:casistica})
shows that
$\xi(T)$ is contained in the rational
polyhedron $M\cup o(p\dot{F})$,
whence
$$\xi(M\cup o(pF))
 \subseteq M\cup o(p\dot{F}).$$
Further, for every
vertex
$v\in\{o,e_1/m_1,\ldots,e_{n}/m_{n}\}$
the last line of (\ref{equation:casistica})
shows that
$\xi(v)=v$.  It follows that
  $\xi(w)=w$ for each
$w\in M\cup o(p\dot{F})$.

In conclusion,  $\xi$ is a
$\mathbb Z$-retraction of
$M\cup o(pF)$ onto $ M\cup o(p\dot{F})$.
 \end{proof}

\bigskip
For the proof of Theorem \ref{theorem:collapsible} below,
given any regular complex $\Lambda$ in $\cube$
with vertex set  ${\mathscr V}=\{v_{1},\ldots,v_{u}\},$
we will construct a $\mathbb Z$-homeomorphic copy
$\Lambda^\perp$  of $\Lambda$  in $[0,1]^u$ in two steps
as follows:

\smallskip
For any $(k-1)$-simplex $T=\conv(v_{i(1)},\ldots,v_{i(k)})$
in $\Lambda$,
 recalling the notation
(\ref{equation:basis})
we first set
\begin{equation}
    \label{equation:perpsimplex}
    T^\perp = \conv(\frac{e_{i(1)}}{\den(v_{i(1)})},\ldots,
\frac{e_{i(k)}}{\den(v_{i(k)})})\subseteq
    [0,1]^u.
\end{equation}
Next we define
\begin{equation}
    \label{equation:perpcomplex}
    {\Lambda}^{\perp} = \{T^\perp\mid T\in \Lambda\}.
 \end{equation}
It follows that
${\Lambda}^{\perp}$  is a regular
complex, whose  symbiotic
relation with $\Lambda$ is given by the following

\begin{lemma}
    \label{Lem-HomeoCanonicRealization}
For every rational polyhedron
$P\subseteq
\cube$ and regular triangulation
$\Lambda$ of $P$ there is a
$\Zed$-homeomorphism $\eta$
of $P$ onto $|{\Lambda}^{\perp}|$
which is linear on each simplex of
$\Lambda$.
  \end{lemma}

\begin{proof}  Letting, as above,
    $\{v_1,\ldots,v_u\}$  be the vertices of
    $\Lambda,$  the proof immediately follows
    from (\ref{equation:perpsimplex})-(\ref{equation:perpcomplex}),
    upon noting that  the map $$f\colon v_j\mapsto
    \frac{e_{j}}{\den(v_{j})},
    \,\,\,(j=1,\ldots,u)$$
    as well as its inverse $f^{-1}$,
    satisfy the hypotheses of
  Lemma
\ref{lemma-ExtensionsToZmorphisms}.
\end{proof}

\begin{theorem}
    \label{theorem:collapsible}
Let $P\subseteq [0,1]^{n}$ be a   polyhedron. Suppose
\begin{itemize}
\item[(i')] $P$ has a
  collapsible  triangulation $\nabla$;
\item[(ii)] $P$ contains a vertex $v$ of $\cube$;
\item[(iii)] $P$ has a strongly regular
triangulation  $\Delta$.
\end{itemize}
Then  $P$ is a $\Zed$-retract of $[0,1]^{n}$.
\end{theorem}

\begin{proof}
    By (iii), $P$ is a rational polyhedron.
 By \cite{bey77rat},  it is no
 loss of generality to assume that
 $\nabla$  is rational.
The desingularization process of \cite[1.2]{mun88}, yields a regular
subdivision $\bar{\nabla}$ of $\nabla$ via finitely many blow-ups.  By
 (i') and
\cite[Theorem 4]{whi}, $\bar{\nabla}$ is
collapsible.  By
 Lemma \ref{Lem-BlowupPreserGCD},
$\bar{\nabla}$ is strongly regular, since so is $\Delta$
by (iii).

Let $v_{1},\ldots,v_{u}$  be the vertices of $\bar{\nabla}$.
Defining the regular
  complex $\bar{\nabla}^{\perp}$
  as in (\ref{equation:perpcomplex}),
  we have a rational polyhedron
  $Q=|\bar{\nabla}^{\perp}|\subseteq [0,1]^u$
  which, by
  Lemma \ref{Lem-HomeoCanonicRealization},
is a $\mathbb Z$-homeomorphic copy
of $P$.
Thus   in view
of Lemma \ref{Lem-Zhomeo-Retrac}
it is sufficient to prove that
\begin{equation}
    \label{equation:claim}
Q \mbox{ is a $\Zed$-retract of } [0,1]^u.
\end{equation}
To this purpose, let us first note that
$\bar{\nabla}^{\perp}$ inherits from
$\bar{\nabla}$ both properties of
strong regularity
and collapsibility.
The vertices of $\bar{\nabla}^{\perp}$ are
$v_{1}^{\perp},\ldots,v_{u}^{\perp}$, where
$v_{i}^{\perp}=\frac{e_i}{\den(v_i)},$ as given by
(\ref{equation:perpsimplex}).  By (ii),
it is no loss of generality to assume $v=v_1$, whence
$\den(v_1)=1$
and $v_{1}^{\perp}=e_1$.

Following Whitehead
\cite[p.248]{whi}, let
\begin{equation}
    \label{equation:firstseq}
\bar{\nabla}^{\perp}=\Delta_0,
\,\,\Delta_1,\ldots,\Delta_m
\end{equation}
be a sequence of
regular
triangulations
such that for each $i=1,\ldots m$, $\Delta_{i}$
is obtained
from $\Delta_{i-1}$ via an elementary collapse, and $\Delta_m$
only consists of the 0-simplex $\{e_1\}$.
For each
$i=1,\ldots, m,\,$ we then have
uniquely determined
simplexes $p_i, F_i,E_i \in \Delta_{i-1}$ such that
 \begin{itemize}
 \item[(a)]  $p_i$ is a vertex of $E_i$ not in $F_i$, and $E_i=p_iF_i$;
 \item[(b)] $F_i$ is a proper face of no other simplex of $\Delta_{i-1}$
 but $E_i$;
 \item[(c)] $\Delta_i=\Delta_{i-1}\setminus \{E_i,F_i\}.$
 \end{itemize}

Letting  $o$ denote the origin in $\mathbb R^u$,
the join
$oQ$ is  {\it star-shaped} at $o$,
\cite[p.38]{hat}, in the sense that
for every $z \in oQ$ the
  set $\conv(o,z)$ is contained
in $oQ$.
The   proof of \cite[Theorem 1.4]{cabmun}
  shows that  $oQ$ is a
$\Zed$-retract
of $[0,1]^u$.
We will construct a sequence
 \begin{equation}
     \label{equation:sequence}
      oQ=|\Lambda_0|\overset{\eta_1}{\longrightarrow}
      |\Lambda_1|\overset{\eta_2}{\longrightarrow}
 \cdots\overset{\eta_m}{\longrightarrow}|\Lambda_m|=
 Q\cup\conv(o,e_1)
     \end{equation}
     of
 $\Zed$-retractions of rational
 polyhedra in $[0,1]^u$, together with
 a $\mathbb Z$-retraction  $\phi$  of
 $Q\cup\conv(o,e_1)$ onto $Q$.
To this purpose, for each  $j=0,\ldots,m\,\,$ we set
\begin{equation}
    \label{equation:detail}
    \Lambda_j = \{o\}\cup \Delta_0
    \cup \{oT\mid T \in \Delta_j\}.
    \end{equation}
Every
 $\Lambda_{j}$ is a regular complex, and
\begin{equation}
    \label{equation:uno}
   \textstyle |\Lambda_j|= |\Delta_0|
    \cup \bigcup \{oT\mid T \in \Delta_j\}=Q
    \cup o\bigcup \{T\mid T \in \Delta_j\}=Q
\cup o|\Delta_j|.
\end{equation}
 Recalling (c) we immediately have
 \begin{equation}
     \label{equation:lambdas1}
     \Lambda_i = \Lambda_{i-1} \setminus
     \{oE_i, oF_i\}\,\,\,\,\, \mbox{ and }\,\,\,\,\,
|\Lambda_{i-1}|=oE_i\cup |\Lambda_i|,
     \end{equation}
     for each $i=1,\ldots,m.\,\,\,$
From (a) we obtain
 \begin{equation}
     \label{equation:lambdas2}
  \textstyle       oE_i\cap|\Lambda_i|=
  \bigcup\{oF\mid F \mbox{ is a facet of }
  E_i\mbox{ and }F\neq F_i\}
  =o(p_i\dot{F_i}),
\end{equation}
for all $i=1,\ldots,m$.
As the reader will recall from
(\ref{equation:dot}), $\dot{F_i}$ denotes the pointset union of facets
of $F_i$.  By (a), $p_i\dot{F_i}$ is the pointset union of
the facets of  $E_i$ different from $F_i$.  We now choose a maximal
simplex $M_i$ of $\Delta_0=\bar{\nabla}^{\perp}$ such that
$E_i\subseteq M_i$.  Since $\Delta_0$ is strongly regular, the
denominators of the vertices of $M_i$ are coprime.

\medskip An application of   Lemma \ref{Lem:empty-m-ahedron}
yields a $\Zed$-retraction $\xi_i$ of $M_i\cup
oE_i =M_i \cup o(p_i F_i)$ onto $M_i\cup
o(p_i\dot{F_i})$.  By (\ref{equation:lambdas1})
and (\ref{equation:lambdas2}), the
map $\eta_i\colon |\Lambda_{i-1}|\rightarrow |\Lambda_i|$
defined by
$$
\eta_i(w)=\left\{\begin{tabular}{ll}
$w$ & if $w\in |\Lambda_i|$\\
$\xi_i(w)$ &  if $w\in oE_i\,$
\end{tabular}
\right.
$$
is   a $\Zed$-retraction of
$|\Lambda_{i-1}|$  onto $|\Lambda_i|$,
as promised in (\ref{equation:sequence}).
The composite
map $\eta=\eta_m\circ\ldots\circ\eta_1\colon oQ
\rightarrow |\Lambda_m|$ is a $\Zed$-retraction
of $oQ$
onto $Q\cup \conv(o,e_1)$.

\medskip
Next  let
 the map $\phi\colon |\Lambda_m|\rightarrow Q$  be defined by
$$
\phi(w)=\left\{\begin{tabular}{ll}
$w$ & if $w\in Q$\\
$e_1$ & if $w\in \conv(o,e_1)$.
\end{tabular}
\right.
$$
Then
$\phi$ is the promised
$\Zed$-retraction of $|\Lambda_m|$
onto $Q$, thus showing that
$Q$ is a $\Zed$-retract of $oQ$.

As already observed,
we have a
$\Zed$-retraction $\mu$ of $[0,1]^u$
onto the star-shaped polyhedron
$oQ$.

Summing up,
the composite map
$\phi\circ\eta\circ\mu$ is a
 $\Zed$-retraction of $[0,1]^u$ onto $Q$,
 as required to settle (\ref{equation:claim}).
 The proof is complete.
\end{proof}

\begin{corollary}
    \label{corollary:tree}
Let $P\subseteq [0,1]^{n}$ be a one-di\-m\-en\-sio\-n\-al
   polyhedron. Then
$P$ is a $\Zed$-retract of $[0,1]^{n}$ iff
$P$ satisfies Conditions (i)-(iii) of Theorem
\ref{theorem:about}. In the present case,
Condition (i) is equivalent to saying that $P$
is connected and simply connected, i.e., $P$ is a tree.
\end{corollary}

\begin{proof}  $P$ necessarily is a rational
    polyhedron.  Let $\nabla$ be a regular
    triangulation of $P$, as given by
   Lemma \ref{lemma:existence}.
In the light of
 Theorems \ref{theorem:about} and \ref{theorem:collapsible}
  with Proposition \ref{proposition:various},
we  have only to check that
$P$ is connected and simply connected
iff
$\nabla$ is collapsible.
$(\Rightarrow):$ Then $P$ contains no
simple closed curve, whence $P$ is a tree.
All triangulations of $P$ are
collapsible. $(\Leftarrow):$
Suppose  $\nabla$ collapses to
its vertex $v$. Then $v$ is a deformation
retract of $P$, and $P$ is contractible,
i.e., $P$ is connected and simply connected.
\end{proof}

\section{Finitely generated  projective unital $\ell$-groups}

 We now apply the results of the previous sections to
  finitely generated projective
unital $\ell$-groups.  In contrast to
what Baker and Beynon proved for
$\ell$-groups in \cite{bak} and \cite{bey77},
finitely generated projective
{\it unital} $\ell$-groups
form a very special
subclass of finitely presented unital $\ell$-groups:

\begin{corollary}
    \label{corollary:parafrasi}
    Suppose
  $(G,u)$ is a finitely presented
  unital $\ell$-group, and write
   $(G,u)\cong \McNn\restrict P$
   for some $n=1,2,\ldots$ and some rational polyhedron
   $P\subseteq \cube$ as given by  Proposition
   \ref{proposition:finpres}.
   \begin{itemize}
\item[(I)]  If  $(G,u)$
is projective then $P$ satisfies Conditions
 (i)-(iii)  of   Theorem \ref{theorem:about}.

  \item[(II)]  If $P$ is one-dimensional
 and satisfies Conditions
 (i)-(iii) then $(G,u)$ is projective.

 \item[(III)] More generally,
 if $P$ satisfies Conditions
 (ii)-(iii) and has a collapsible
 triangulation then
 $(G,u)$ is projective.

 \end{itemize}
\end{corollary}

\begin{proof}
 (I) In the light of   Theorems \ref{theorem:about} and
 \ref{Thm_Projective_Retraction}  we must only prove that $P$ is a
 $\mathbb Z$-retract of $\cube$.  By
  Lemma \ref{Lem-Zhomeo-Retrac},
 it suffices to settle the following

 \medskip
 \noindent {\it Claim:} If  a
 rational polyhedron $Q\subseteq [0,1]^{m}$
satisfies
$\McNn\restrict P \cong \McN_m\restrict Q$
then $P$ and $Q$ are $\mathbb
 Z$-homeomorphic.

 \smallskip
To this purpose, let $\iota\colon \McNn\restrict
P\cong\McN_m\restrict Q$ be
a unital $\ell$-isomorphism. Let
 $\xi_{1},\ldots,\xi_{n}\colon\cube\rightarrow [0,1]$ be the coordinate
 maps.  Each element $\xi_{i}\restrict P\in\McNn\restrict P$
 is sent by $\iota$
 to some element $h_{i}\restrict Q$ of $\McN_m\restrict Q.$ Since
each   $\xi_i$ belongs to the unit interval of  $\McNn\restrict P$,
then  $h_{i}$ belongs to the unit interval of
 $\McN_m\restrict Q,$
 i.e., the range of $h_{i}$ is contained in the unit interval
 $[0,1]$.
Then the map  $\eta\colon[0,1]^{m} \to
 [0,1]^{n}$ defined by
$$\eta(x_1,\ldots,x_m)=(h_{1}(x_1,\ldots,x_m),\ldots,h_{n}(x_1,\ldots,x_m)),
\,\,\,\forall (x_1,\ldots,x_m)\in [0,1]^m,$$
is a $\Zed$-map.
Let  $f$ be an arbitrary function in  $\McNn$.
Arguing by induction on the number
 of operations in $f$ in the light of   Proposition
 \ref{proposition:free}, we get
\begin{equation}
    \label{eq:iotaeta}
 \iota(f\restrict P)=(f\circ \eta)\restrict Q.
\end{equation}
Since $\iota$ is a unital $\ell$-isomorphism,
  $$ f\restrict P=0
 \Leftrightarrow\iota(f\restrict P)=0\Leftrightarrow
 f\circ\eta\restrict Q =0\Leftrightarrow f\restrict \eta(Q)=0.
 $$
  By \cite[Proposition 4.17]{mun86}, $P=\eta(Q)$.
 Interchanging the roles of $\iota$ and $\iota^{-1}$ we obtain a
 $\Zed$-map $\mu\colon \cube\rightarrow [0,1]^m$ such that
 $\iota^{-1}(g\restrict Q)=(g\circ \mu)\restrict P \text{ and }
 \mu(P)=Q.$
 By
 (\ref{eq:iotaeta}),
 $\,\,
 f\restrict P=f\circ (\eta\circ \mu)\restrict P\text{ and }g\restrict
 Q=g\circ(\mu\circ\eta)\restrict Q,
 $ for each $f\in\McNn$ and
 $g\in\McN_m$.  Again by \cite[Proposition 4.17]{mun86},
 the composition
 $\eta\circ\mu$ is the identity map on $P$,
 and $\mu\circ\eta$ is the
 identity map on $Q$.
 Thus  $P$ and $Q$ are
 $\Zed$-homeomorphic, as required to settle our claim
 and also to complete the proof of (I).

  (II) From  Corollary
 \ref{corollary:tree} and Theorem \ref{Thm_Projective_Retraction}.

  (III) From   Theorems \ref{theorem:collapsible} and
 \ref{Thm_Projective_Retraction}.
\end{proof}

\bigskip
Our final result in this paper,
Theorem \ref{theorem:lim} below,  will
give (intrinsic) necessary and sufficient conditions
for $(G,u)$ to be finitely generated projective,
in terms of the spectral properties of
$G$.
To this purpose, we denote by
$$\maxspec({G})$$
the set of maximal
 ideals of $G$,
 equipped with the {\it spectral} topology, \cite[\S
10]{bigkeiwol}, \cite[5.7]{gla}: a basis of closed sets for
$\maxspec(G)$ is given by all sets of the form $ \{\mathfrak
p \in \maxspec(G) \mid a\in \mathfrak p\}, $ where $a$ ranges over all
elements of $G$.

A maximal ideal
$\mathfrak m$ is {\it discrete} if the ordering of
the totally ordered quotient
$(G,u)/\mathfrak m$  is discrete  (non-dense).
In this case, by the Hion-H{\"o}lder theorem \cite[p.45-47]{fuc},
\cite[2.6]{bigkeiwol},
$(G,u)$ is unitally $\ell$-isomorphic to
$(\mathbb Z\frac{1}{n},1)$
for a unique integer $n\geq 1$,
called the {\it rank} of  $\mathfrak m$
and denoted $\rho(\mathfrak m)$.
In case  $\mathfrak m$ is not discrete
we set
\begin{equation}
\label{equation:infinito}
\rho(\mathfrak m)=+\infty\,\,\,\mbox{ and }\,\,\,
\gcd(n,+\infty)=+\infty,  \forall n=1,2,\ldots.
\end{equation}

 \begin{lemma}
    \label{lemma:hoelder}
For every $n=1,2,\ldots$ and nonempty closed set
$X\subseteq [0,1]^{n}$   we have:
    \begin{itemize}

        \item[(a)]  The map $\alpha\colon x\in X\mapsto
        {\mathfrak m}_{x}=\{f\in
        \McNn\restrict X
        \mid f(x)=0\}$ is a homeomorphism
        of $X$ onto
        $\maxspec(\McNn\restrict X)$.
	The inverse map
	sends every
        $\mathfrak m\in \maxspec(\McNn\restrict X)$ to the only
member $x_{\mathfrak m}$ of the set  $\bigcap\{g^{-1}(0)\mid
        g\in \mathfrak m\}$.

\smallskip
    \item[(b)] For every
    $\mathfrak{m}\in \maxspec(\McNn\restrict X)$  there is a unique pair
    $(\iota_{\mathfrak m},R_{\mathfrak m})$, where
    $R_{\mathfrak m}$ is a unital $\ell$-subgroup of
    $(\mathbb R,1)$, and $\iota_{\mathfrak m}$
    is a unital $\ell$-isomorphism of the quotient
    $\maxspec(\McN\restrict X)/\mathfrak m$ onto $R_{\mathfrak m}$.
     For every $x\in X$
and   $f\in  \McNn\restrict X$,
$\,\,\,f(x)=\iota_{{\mathfrak m}_{x}}(f/{\mathfrak m}_{x})$.

\smallskip
\item[(c)]  Suppose  $x\in X$ and  $\mathfrak m=\alpha(x)$.
Then $x$ is rational iff $\mathfrak m$ is discrete.
If  this is the case,  $\rho(\mathfrak m)=
\den(x)$.
        \end{itemize}
\end{lemma}

\begin{proof} The proof of
    (a) follows from a classical
    result due to Yosida \cite{yos}, because
    the functions in
    $\McNn\restrict X$ separate points,
    \cite[4.17]{mun86}.  (See \cite[8.1]{mun86} for further
    details).
   (b)  is a reformulation of the time-honored
   Hion-H{\"o}lder theorem
 \cite[p.45-47]{fuc},
\cite[2.6]{bigkeiwol}.
 Finally,  (c) follows from (a) and (b).
 \end{proof}

\begin{theorem}
\label{theorem:lim}
Let $(G,u)$ be a unital $\ell$-group.
\begin{enumerate}
    \item
The following  conditions are  necessary for
$(G,u)$ to be finitely generated projective:

\begin{itemize}
    \item[(A)] $(G,u)$ is finitely presented.

    \item[(B)]  For every  discrete maximal ideal
    $\mathfrak m$ of
$G$, and  open neighborhood $N$ of
$\mathfrak m$ in $\maxspec(G)$,
there is $\mathfrak n\in N$ such that
$\gcd(\rho(\mathfrak n),
	\rho(\mathfrak m))=1.
	$

    \item[(C)] $G$ has a maximal
ideal of rank $1$.

    \item[(D)] The topological space
    $\maxspec(G)$ is compact Hausdorff,
    metrizable, finite-dimensional and contractible.

\end{itemize}
\item   If $\maxspec(G)$ is one-dimensional,
the four conditions (A)-(D) are also sufficient
for $(G,u)$ to be finitely generated projective.
Actually, Condition (D) can be replaced
by the requirement that
$\maxspec(G)$ is connected and
simply connected.

\item  More generally,  if $(G,u)$
satisfies Conditions (B)-(C),  and
$(G,u)\cong\McNn\restrict P$
for some rational polyhedron $P$ having
a collapsible  triangulation,
then $(G,u)$ is finitely generated projective.
\end{enumerate}
\end{theorem}

\begin{proof}
(1)  Condition (A) holds by
  Proposition \ref{proposition:finpres}.

 To prove Condition  (B), by (A) jointly with   Corollary
  \ref{corollary:parafrasi}(I) we can write
  $(G,u)=\McNn\restrict P$  for some rational
  polyhedron  $P\subseteq \cube$ satisfying
 Conditions (i)-(iii) of  Theorem
  \ref{theorem:about}.
  Using the  homeomorphism
  $\alpha$ of
  Lemma \ref{lemma:hoelder},
  ranks of discrete maximal ideals
  of $G$ coincide with denominators of their
  corresponding  rational points in $P$.
If $P$ is a singleton, then by Condition (ii)
  it coincides with some
  vertex of $\cube$  and we have nothing to prove.
Otherwise, let
  $\Delta$ be a strongly regular triangulation
  of $P$ as given by Condition (iii).
  Let $x$  be a rational point of $P$.
  The proof of Theorem \ref{theorem:about}
  shows that every open neighborhood
 of $x$ contains rational points  $q$  of arbitrarily
  large prime denominator, whence  $\gcd(\den(x),\den(q))=1$,
  and  (B) is proved.

  Using $\alpha$ we see that
  Condition (C) holds, because $P$
  satisfies  Condition (ii) of  Theorem
  \ref{theorem:about}.

To conclude the proof of (1),  we must show that
 $\maxspec(G)$  has
all the properties listed in Condition (D).
In the light of   Lemma \ref{lemma:hoelder},
this is equivalent to checking
that $P$ has all these
  properties. (The invariance
  of contractibility under
  homeomorphisms follows, e.g., from
 Proposition \ref{proposition:various}.)
The first three properties
  are trivially verified.
  $P$ is contractible
  because it satisfies Condition (i)
of   Theorem
  \ref{theorem:about}.

(2) Since $(G,u)$ satisfies Condition (A),  it can be identified
with   $\McNn\restrict P$, for some  $n=1,2,\ldots$
and rational polyhedron $P\subseteq
\cube$.  The homeomorphism $\alpha$ of  Lemma
\ref{lemma:hoelder} again ensures that
ranks of discrete maximal ideals of
$G$ coincide with denominators of their corresponding rational points
of   $P$.  Thus  Condition (C)
is to the effect that  $P$
must contain some vertex of $\cube,$
whence  $P$ satisfies     Condition (ii) of  Theorem
\ref{theorem:about}.

We next prove that
 {\it every}
  regular triangulation $\Delta$
  of $P$ is strongly regular. By
   Remark \ref{remark:trebis}
 this is an equivalent reformulation of
 Condition (iii).  Suppose  $\Delta$  is a  counter-example,
  and let $T$ be a maximal simplex of $\Delta$
  such that the gcd of the denominators of the
  vertices of $T$ is $d>1.$  Pick a rational
  point $q$ in the relative interior of $P$
  and observe that, by
 Proposition \ref{proposition:cauchy},
  $d$ is a divisor of  $\den(q)$.
  By the assumed maximality
  of $T$, for every rational point $q'$
 in a suitably small open neighborhood of  $q$,
  $d$ is a divisor of $\den(q')$.
The maximal ideal  $\alpha(q)$ of $G$
falsifies the assumed Condition (B).
We have shown that $P$ satisfies
  Condition (iii) of Theorem
\ref{theorem:about}.

Further,  $P$ satisfies  Condition (i)
because its homeomorphic copy
$\maxspec(G)$ is contractible,
by Condition (D).

Having  thus shown  that $P$
satisfies Conditions (i)-(iii)
of  Theorem \ref{theorem:about}, an
application of   Corollary
  \ref{corollary:parafrasi}(II)
  proves the first statement in (2).

For the second statement, since
$G$ has an order unit,
$\,\,\maxspec(G)$ is a nonempty compact Hausdorff space,
(for a proof see \cite[10.2.2]{bigkeiwol}, where
the order unit is called ``unit\'e forte'').
The   homeomorphism
$\alpha$ of   Lemma \ref{lemma:hoelder}
ensures that there
is no ambiguity in defining the dimension
of the compact Hausdorff
metrizable space $P$ and of its
homeomorphic copy  $\maxspec(G)$.
 Condition (A)  ensures that
  $\maxspec(G)$ is finite-dimensional and
 metrizable.  Thus
Condition (D) equivalently states that
$\maxspec(G)$ is contractible.
By   Proposition \ref{proposition:various},
this  is in turn equivalent to
stating  that
the one-dimensional
space  $\maxspec(G)$  is connected and simply connected.
This completes the proof of  (2).

(3)  By   Proposition \ref{proposition:finpres},
$(G,u)$ is finitely presented.
Given the map $\alpha$
of  Lemma \ref{lemma:hoelder},    Condition (C)
is equivalent to stating
that $P$ satisfies Condition (ii)
of  Theorem \ref{theorem:about}.
Condition (iii)
now follows from Condition (B) arguing as in (2) above.
An application of   Corollary
  \ref{corollary:parafrasi}(III) concludes the proof.
\end{proof}


  \subsection*{Acknowledgment}
  The authors are grateful to
  Professor Marco Grandis and to Dr.
  Bruno Benedetti for introducing them
  to the main results of algebraic topology
  used in this paper.  We also thank
  Professor Vincenzo
  Marra for drawing our attention to reference
  \cite{bey77rat}.


\end{document}